\documentclass[12pt]{article}

\usepackage[margin=1.20in]{geometry}

\usepackage{amssymb,amsmath}
\usepackage{amsthm,enumitem}
\usepackage{hyperref}
\usepackage{mathrsfs}
\usepackage{textcomp}
\hypersetup{
    colorlinks,%
    citecolor=black,%
    filecolor=black,%
    linkcolor=black,%
    urlcolor=black
}

\usepackage{verbatim}

\usepackage{sectsty}

\makeatletter

\newdimen\bibspace
\renewenvironment{thebibliography}[1]{%
 \section*{\refname %or \bibname if you use ``book'' as the documentclass
       \@mkboth{\MakeUppercase\refname}{\MakeUppercase\refname}}%
     \list{\@biblabel{\@arabic\c@enumiv}}%
          {\settowidth\labelwidth{\@biblabel{#1}}%
           \leftmargin\labelwidth
           \advance\leftmargin\labelsep
           \itemsep\bibspace
           \parsep\z@skip     %
           \@openbib@code
           \usecounter{enumiv}%
           \let\p@enumiv\@empty
           \renewcommand\theenumiv{\@arabic\c@enumiv}}%
     \sloppy\clubpenalty4000\widowpenalty4000%
     \sfcode`\.\@m}
    {\def\@noitemerr
      {\@latex@warning{Empty `thebibliography' environment}}%
     \endlist}

\makeatother

\makeatletter

\newtheorem{thm}{Theorem}[section]
\newtheorem{lem}[thm]{Lemma}
\newtheorem{prop}[thm]{Proposition}

\def\XXint#1#2#3{{\setbox0=\hbox{$#1{#2#3}{\int}$}
  \vcenter{\hbox{$#2#3$}}\kern-.5\wd0}}

\newcommand{\al}{\alpha}                \newcommand{\lda}{\lambda}
\newcommand{\om}{\Omega}                \newcommand{\pa}{\partial}
           \newcommand{\ud}{\mathrm{d}}
\newcommand{\be}{\begin{equation}}      \newcommand{\ee}{\end{equation}}
                 
\newcommand{\Lda}{\Lambda}              
\newcommand{\R}{\mathbb{R}}

\begin{document}

\title{\textbf{Regularity of solutions to the Dirichlet problem for fast diffusion equations}\bigskip}

\author{\medskip  Tianling Jin\footnote{T. Jin was partially supported by Hong Kong RGC grant GRF 16306320 and NSFC grant 12122120.}\quad and \quad
Jingang Xiong\footnote{J. Xiong was is partially supported by  the National Key R\&D Program of China No. 2020YFA0712900, and NSFC grants 11922104 and 11631002.}}

\date{\today}

\maketitle

\begin{abstract} 
We prove global H\"older gradient estimates for bounded positive weak solutions of fast diffusion equations in smooth bounded domains with the homogeneous Dirichlet boundary condition, which then lead us to establish their optimal global regularity. This solves a problem raised by Berryman and Holland in 1980. 

\medskip


\medskip


\end{abstract}

\section{Introduction}

Let $\om$ be a bounded smooth domain in $\R^n$ with $n\ge 1$ and let $p\in (1,\infty)$. Consider the fast diffusion equation
\be \label{eq:main}
\pa_t u^p -\Delta u=0 \quad \mbox{in }\om \times (0,\infty)
\ee
with the homogeneous Dirichlet  boundary condition
\be \label{eq:main-d}
u=0 \quad \mbox{on }\pa \om \times (0,\infty)
\ee
and the initial condition 
\be \label{eq:initial}
u(x,0)= u_0 (x)\ge 0.
\ee 
In 1962, Sabinina \cite{S2} proved that the above Cauchy-Dirichlet problem admits a unique nonnegative weak solution if $u_0\in C_0^1(\om)$, shortly after the seminal paper of Oleinik-Kalashnikov-Chzou \cite{OKC} and Sabinina \cite{S1} on the porous medium equation.  Following \cite{OKC, S2}, 
we say that $u$ is a weak solution of \eqref{eq:main}-\eqref{eq:initial} if $u\in L^2_{loc}([0,\infty); H^1_0(\Omega))$, $u\ge 0$, $u^p\in C([0,\infty);L^{1}(\Omega))$, and satisfies 
\[
\int_{t_1}^{t_2}\int_{\Omega} \Big(u^p\pa_t\varphi -\nabla u\cdot\nabla\varphi\Big)\,\ud x\ud t= \int_{\om}u(x, t_2)^p \varphi \,\ud x- \int_{\om}u(x, t_1)^p \varphi \,\ud x
\] 
for any $0\le t_1< t_2<\infty$ and $\varphi\in C^1(\overline\Omega\times[0,\infty))$ which vanishes on $\pa \om \times [t_1,t_2]$.

The space $ C_0^1(\om)$ for the initial data can be enlarged to $ H_0^1(\om) \cap L^\infty(\om)$, and then, by the maximum principle, the solution $u$ will be bounded. If $u(\cdot,0)\not\equiv 0$, then there exists the so-called finite extinction time, i.e., a time $T^*>0$ such that $u(\cdot, t)\equiv 0$ for all $t\ge T^*$, and $u(\cdot, t)>0$ in $\Omega$ for $t\in(0,T^*)$. 
See Chapter 4 of Daskalopoulos-Kenig \cite{DaK} and Chapters 5-7 of V\'azquez \cite{Vaz} for systematic studies of the above Cauchy-Dirichlet problem (including the porous medium equations corresponding to $0<p<1$). 

In the pioneering work \cite{BH}, Berryman and Holland studied the extinction behavior of solutions to the Cauchy-Dirichlet problem \eqref{eq:main}-\eqref{eq:initial}. Their work was motivated by the Wisconsin octupole experiments of Drake-Greenwood-Navratil-Post \cite{DGNP} on anomalous diffusion of hydrogen plasma across a purely poloidal octupole magnetic field, where the density of the plasma satisfies the fast diffusion equation \eqref{eq:main} with $p=2$.  The experiments in \cite{DGNP} have shown that after a few milliseconds the solution $u$ evolves into a fixed shape which then decays in time. Berryman-Holland \cite{BH} proved this phenomenon along a sequence of time under the regularity assumption
\be \label{eq:bh}
\partial_t u, \nabla u, \nabla\partial_t u, \nabla^2 u\in C(\overline \om\times (0,T^*)).
\ee
In their final section of  conclusions and conjectures, they listed \eqref{eq:bh} as the first unsolved problem.  

The difficulty of this regularity problem lies in the singular parabolic structure near the boundary $\partial\Omega$, where $u=0$.  Sacks \cite{Sacks} and DiBenedetto \cite{DiBenedetto} proved that  bounded positive solutions to \eqref{eq:main}-\eqref{eq:main-d} are continuous up to the boundary. 
Global H\"older continuity was later proved by Chen-DiBenedetto \cite{CDi}. See also Chapter IV of DiBenedetto \cite{DiBenedettobook}. Kwong \cite{KwongY, KwongY2} and  DiBenedetto-Kwong-Vespri \cite{DKV} proved spatial Lipschitz regularity, that is, for every  $\delta>0$, there exists $C>0$ such that
\be \label{eq:DKV}
\frac{1}{C}\le \frac{u}{d}\le C \quad \mbox{in }\om\times[\delta,T^*-\delta],
\ee
where $d(x)=\mathrm{dist}(x,\pa \om)$. 
Various $L^\infty$ estimates, Harnack inequalities and more continuity estimates can be found in, e.g.,  Aronson-B\'enilan \cite{AB},  Caffarelli-Evans \cite{CE},   Herrero-Pierre \cite{HP}, Dahlberg-Kenig \cite{DKenig88},  DiBenedetto-Kwong \cite{DK}, DiBenedetto-Gianazza-Vespri \cite{DGV,DGV-b}, Lu-Ni-V\'azquez-Villani \cite{LNVV} and Bonforte-V\'azquez \cite{BV}.

In our previous paper \cite{JX19}, we proved the regularity property \eqref{eq:bh} asked by Berryman and Holland in the Sobolev subcritical and critical regimes, and also obtained their optimal regularity.  The dividing of $p$ according to the Sobolev embedding is as follows: (a) \textit{Subcritical: $1<p<\infty$ and $n=1,2$, or $1<p<\frac{n+2}{n-2}$ and $n\ge 3$}; (b) \textit{Critical: $p=\frac{n+2}{n-2}$ and $n\ge 3$}; (c) \textit{Supercritical: $p>\frac{n+2}{n-2}$ and $n\ge 3$.} 
We developed a regularity theory for  a linear singular parabolic equation and used it to bootstrap the regularity of weak solutions of the fast diffusion equation.  With the bootstrap argument, to establish \eqref{eq:bh} and the optimal regularity, it reduces to proving the
\be \label{eq:bootstrap-stp}
\mbox{H\"older continuity of } \frac{u}{d} \mbox{ on }\overline\Omega\times(0,T^*).
\ee  

In the setting of linear uniformly parabolic equations,  boundary estimates of type \eqref{eq:bootstrap-stp}  have been studied in, e.g., Krylov \cite{Kr}, Fabes-Garofalo-Salsa \cite{FGS} and Fabes-Safonov \cite{FabS}. For boundary estimates of solutions to certain degenerate or singular nonlinear  parabolic equations related to \eqref{eq:main}, we refer to the recent papers  Kuusi-Mingione-Nystr\"om \cite{KMN}, Avelin-Gianazza-Salsa \cite{AGS} and the references therein. In particular, Avelin-Gianazza-Salsa \cite{AGS} studied Carleson estimates and boundary Harnack inequality for local solutions of \eqref{eq:main} and \eqref{eq:main-d} with $1<p<\frac{n}{(n-2)^+}$, and it is mentioned in its Remark 4.10 that H\"older estimates for the gradient of the solutions, as well as \eqref{eq:bootstrap-stp}, were still unknown. 

We proved \eqref{eq:bootstrap-stp} in \cite{JX19} for $p$ in the  subcritical and critical regimes. A geometric type structure of the fast diffusion equation, inspired by  the Yamabe flow, was used to prove high integrabilities of $\pa_t u/u$, and then H\"older estimates of $D_x u$ implying \eqref{eq:bootstrap-stp}. The quantity $\pa_t u/u$ was also studied in Aronson-B\'enilan \cite{AB}. When $p=\frac{n+2}{n-2}$ and $n\ge 3$, it corresponds to the scalar curvature of the metrics along the Yamabe flow. The Yamabe flow was introduced by Hamilton \cite{Ham}, and its convergence on smooth compact Riemannian manifolds has been proved by  Ye \cite{Y}, Schwetlick-Struwe  \cite{SS} and Brendle \cite{Br05, Br07}.

In this paper, we establish  \eqref{eq:bootstrap-stp} for all $1<p<\infty$ in all dimensions. This is achieved by  applying  the iteration method of De Giorgi \cite{DG} to a very degenerate and very singular nonlinear parabolic type equation. Consequently, we completely solve the regularity problem \eqref{eq:bh} of Berryman and Holland \cite{BH}.

\begin{thm}\label{thm:main} Suppose $n\ge 1$ and $1<p<\infty$.  Let $u$ be a bounded nonnegative weak solution of \eqref{eq:main} and \eqref{eq:main-d}, and $T^*>0$ be its extinction time. 
 \begin{itemize}
\item If $p$ is an integer, then $u\in C^\infty(\overline \om\times (0,T^*))$. 
\item If $p$ is not an integer, then $u(x,\cdot)\in C^\infty(0,T^*)$ for every $x\in\overline\Omega$, and $ \pa_t^l u(\cdot,t) \in C^{2
+p}(\overline \om)$ for all $l\ge 0$ and all $t\in(0,T^*)$.  
\end{itemize}
\end{thm}

A few remarks are as follows.

(i) The regularity in Theorem \ref{thm:main}  is optimal, which can be seen from the separable solutions 
\[
U(x,t):=\left(\frac{p-1}{p}\right)^{\frac{1}{p-1}} (T^*-t)^{\frac{1}{p-1}} S(x),
\] where $S$ is a positive solution of $-\Delta S=S^p$ in $\om$ and $S=0$ on $\pa \om$. Such $S$ exists if, for example,  $p$ is subcritical or $\om$ is an annulus.

 (ii) Kwong \cite{KwongY3}  obtained the same convergence result as in Holland-Berryman \cite{BH} bypassing the regularity  assumption \eqref{eq:bh}. Feireisl-Simondon \cite{FS} proved the  energy inequality
\[
\int_\om |\nabla u(\cdot,t_2)|^2- \int_\om |\nabla u(\cdot,t_1)|^2 + 2p\int^{t_2}_{t_1} \int_{\om }u^{p-1} |\pa_t u|^2  \,\ud x\ud t \le 0
\]
for $ 0<t_1<t_2<T^*$ without the regularity assumption \eqref{eq:bh}, which was then used to strengthen the convergence to be uniform in time, instead of just along a sequence of time. By Theorem \ref{thm:main}, this energy inequality now becomes an equality, since the derivation 
\[
\frac{\ud }{\ud t} \int_\om |\nabla u(\cdot,t)|^2= 2 \int_{\om} \nabla u \nabla \pa_t u = -2p \int_{\om }u^{p-1} |\pa_t u|^2
\]
by Berryman-Holland \cite{BH} is justified.

(iii)  In the subcritical regime, Bonforte-Grillo-V\'azquez \cite{BGV} proved the uniform convergence of the relative error, and Bonforte-Figalli \cite{BFig} proved the sharp exponential decay of the relative error on generic domains; see Akagi \cite{Akagi} for another proof. In our earlier papers  \cite{JX19, JX20}, we proved the sharp exponential decay of the relative error in the $C^2$ topology on generic domains, and polynomial decay of the relative error in the $C^2$ topology on all smooth domains. Recently, Choi-McCann-Seis \cite{CMS} proved that the relative error either decays exponentially with the sharp rate or else decays algebraically at a rate $1/t$ or slower. They also obtained higher order asymptotics of the relative error.

(iv) For the porous medium equation, that is the equation \eqref{eq:main} with $0<p<1$, the regularity of nonnegative solutions and the regularity of their free boundaries have been studied by, e.g., Aronson \cite{A, A70}, Aronson-B\'enilan \cite{AB},  Caffarelli-Friedman \cite{CF}, Caffarelli-V\'azquez-Wolanski \cite{CVW}, Caffarelli-Wolanski \cite{CW},  Daskalopoulos-Hamilton \cite{DH}, Koch \cite{Koch} and Daskalopoulos-Hamilton-Lee \cite{DHL}.

This paper is organized as follows. Section \ref{sec:inequality} is devoted to the inequalities that are needed for the De Giorgi iteration in our setting. In Section \ref{sec:holder}, we carry out the full details of the De Giorgi iteration for a degenerate and singular parabolic equation, and prove H\"older estimates of its solutions, which will lead to \eqref{eq:bootstrap-stp}. In Section \ref{sec:regularity}, we prove Theorem \ref{thm:main}.

\bigskip

\noindent \textbf{Acknowledgements:}  This work is dedicated to Professor YanYan Li on the occasion of his 60th birthday.    Both authors are deeply grateful to him for his thoughtful and thorough guidance and numerous supports.

\section{Some weighted inequalities}\label{sec:inequality}

We start with a weighted Sobolev inequality. Let $x=(x',x_n)\in\R^n$. 
\begin{prop}\label{prop:weightedsobolevinRn}
Let $n\ge 1$ and $p>0$. There exists $C=C(n,p)>0$ such that
\[
\left(\int_{\R^n}|u(x)|^{\frac{2(n+p+1)}{n+p-1}}|x_n|^{\frac{2n}{n+p-1}}\,\ud x\right)^{\frac{n+p-1}{n+p+1}}\le C \int_{\R^n}|\nabla u(x)|^2|x_n|^2\,\ud x
\]
for every Lipschitz continuous function $u$ on $\R^n$ with compact support.
\end{prop}

\begin{proof} It follows directly from Corollary 2 in Section 2.1.7 of Maz'ya \cite{Mazya}. Although it is stated for $n\ge 3$ there, its proof also applies to obtain the above inequality when $n=1,2$. See also  Proposition 1 of Castro \cite{Castro} for $n=1$ and Corollary 1.2 of Dou-Sun-Wang-Zhu \cite{DSWZ} for $n\ge 2$.
\end{proof}

For $R>0$, we denote
\[
B_R^+=B_R\cap\{x_n>0\},\quad \partial' B_R^+=(\partial B_R^+)\cap \{x_n=0\}, \quad \partial'' B_R^+=(\partial B_R^+)\cap \{x_n>0\}.
\]
For a  function $u$ that is Lipschitz continuous on $\overline{B_R^+}\times(-T,0]$, and $p>0$, we define
 \[
  \|u\|_{V_0^1(B_R^+\times (-T,0])}=\sqrt{ \sup_{-T<t<0} \int_{B_R^+ }u(x,t)^2 x_n^{p+1}\,\ud x +  \int_{-T}^0\int_{B_R^+}  |\nabla u(x,t)|^2x_n^2\,\ud x\ud t}.
\]
Then we have the following Sobolev inequality in the parabolic setting.

\begin{lem}\label{lem:weightedsobolev} 
Let $n\ge 1$, $p> 0$, $R>0$ and $T>0$.
For every function $u$ that is Lipschitz continuous on $\overline{B_R^+}\times(-T,0])$ and vanishes on $\partial'' B_R^+\times (-T,0]$, we have
 \[
 \Big(\int_{B_R^+ \times (-T,0]}  |u(x,t)|^{2\chi}x_n^2\,\ud x \ud t \Big)^{\frac{1}{\chi}} \le C \|u\|_{V_0^1(B_R^+\times (-T,0])}^2,
 \]
 where $\chi =\frac{n+p+3}{n+p+1} >1$, and $C>0$ depends only on $n$ and $p$.
\end{lem}

\begin{proof} 
By the H\"older inequality, we have
\begin{align*}
\int_{B_R^+  }x_n^2|u|^{2\chi } \,\ud x 
&=\int_{B_R^+  }|u|^{2} x_n^{2- (p+1)(\chi-1)} |u|^{2(\chi-1)} x_n^{(p+1)(\chi-1)} \,\ud x  \\
&\le \Big(\int_{B_R^+} |u|^{\frac{2}{2-\chi}} x_n^\frac{2- (p+1)(\chi-1)}{2-\chi} \,\ud x\Big)^{2-\chi} \Big( \int_{B_R^+ } u^2 x_n^{p+1}\,\ud x\Big)^{\chi-1} \\
&=\Big(\int_{B_R^+}|u|^{\frac{2(n+p+1)}{n+p-1}}|x_n|^{\frac{2n}{n+p-1}}\,\ud x\Big)^{\frac{n+p-1}{n+p+1}}
\Big( \int_{B_R^+ } u^2 x_n^{p+1}\,\ud x\Big)^{\frac{2}{n+p+1}}. 
\end{align*}
By making an even extension of $u$ for $x_n<0$, and applying Proposition \ref{prop:weightedsobolevinRn}, we obtain
\[
\int_{B_R^+  }|u|^{2\chi }x_n^2 \,\ud x \le C \Big(\int_{B_R^+  } |\nabla u|^2x_n^2\,\ud x\Big)   \Big( \int_{B_R^+ } u^2 x_n^{p+1}\,\ud x\Big)^{\frac{2}{n+p+1}} .
\]
Integrating the above inequality in $t$, we have
\begin{align*}
&\Big(\int_{-T}^0\int_{B_R^+  }|u(x,t)|^{2\chi} x_n^2\,\ud x \ud t\Big)^{\frac{1}{\chi}}\\
& \le C  \sup_{-T<t<0}\Big( \int_{B_R^+  }u^2 x_n^{p+1}\,\ud x\Big)^{\frac{2}{n+p+3}} \Big(\int_{B_R^+ \times[-T,0]} |\nabla u|^2x_n^2 \,\ud x \ud t  \Big)^{\frac{n+p+1}{n+p+3}}\\&
\le C \Big( \int_{B_R^+ \times[-T,0]}x_n^2 |\nabla u|^2 \,\ud x \ud t  + \sup_{-T<t<0} \int_{B_R^+  }u^2 x_n^{p+1}\,\ud x \Big),
\end{align*}
where we have used Young's inequality in the last inequality.
\end{proof}

The next inequality is a Poincar\'e inequality, which follows from Theorem 1.1 and Theorem 1.2 of Chua-Wheeden \cite{ChuaW}.
\begin{prop}\label{prop:weightedpoincare}
Let $n\ge 1$, $p>0$, and $r>0$. Then
\[
\int_{B_r^+}|u(x)-\bar u|x_n^{p+1}\,\ud x\le r \int_{B_r^+}|\nabla u(x)|x_n^{p+1}\,\ud x,
\]
for all Lipschitz continuous function $u$ on $B_r^+$, where 
\[
\bar u=\frac{\int_{B_r^+} u x_n^{p+1}\,\ud x}{\int_{B_r^+} x_n^{p+1}\,\ud x}.
\]
\end{prop}
The last inequality is a De Giorgi type isoperimetric inequality.

\begin{prop}\label{prop:degiorgiisoperimetricelliptic}
Let  $n\ge 1$, $p>0$, $k<\ell, r>0$ and $u$ be a Lipschitz continuous function on $B_r^+$. There exists a positive constant $C$ depending only on $n$ and $p$ such that
\begin{align*}
&(\ell-k)\int_{\{u\ge l\}\cap B_r^+}  x_n^{p+1}\,\ud x\int_{\{u\le k\}\cap B_r^+} x_n^{p+1}\,\ud x\\
&\quad\le C r^{n+p+2} \left(\int_{\{k<u<\ell\}\cap B_r^+}|\nabla u|^2x_n^2\,\ud x\right)^{\frac 12} \left(\int_{\{k<u<\ell\}\cap B_r^+}x_n^{2p}\,\ud x\right)^{\frac 12}.
\end{align*}
\end{prop}
\begin{proof}
Let
\[
v=\sup(k, \inf(u,\ell))-k,\quad \bar v=\frac{\int_{B_r^+} v x_n^{p+1}\,\ud x}{\int_{B_r^+} x_n^{p+1}\,\ud x}.
\]
Then
\begin{align*}
\int_{\{v=0\}\cap B_r^+}\bar v x_n^{p+1}\,\ud x%&= \int_{\{v=0\}\cap B_r^+}|v(x)-\bar v| x_n^{p+1} \,\ud x\\
&\le \int_{B_r^+}|v(x)-\bar v| x_n^{p+1}\,\ud x\\
&\le r \int_{B_r^+}|\nabla v(x)|x_n^{p+1} \,\ud x\\
&= r \int_{\{k<u<\ell\}\cap B_r^+}|\nabla u(x)|x_n^{p+1} \,\ud x\\
&\le r \left(\int_{\{k<u<\ell\}\cap B_r^+}|\nabla u|^2x_n^2\,\ud x\right)^{\frac 12} \left(\int_{\{k<u<\ell\}\cap B_r^+}x_n^{2p}\,\ud x\right)^{\frac 12}.
\end{align*}
On the other hand, we have
\begin{align*}
\int_{\{v=0\}\cap B_r^+}\bar v x_n^{p+1}\,\ud x&=\frac{\int_{B_r^+} v x_n^{p+1}\,\ud x}{\int_{B_r^+} x_n^{p+1} \,\ud x}\cdot \int_{\{u\le k\}\cap B_r^+} x_n^{p+1} \,\ud x\\
& \ge \frac{(\ell-k)\int_{\{u\ge l\}\cap B_r^+}  x_n^{p+1} \,\ud x}{\int_{B_r^+} x_n^{p+1} \,\ud x}\cdot \int_{\{u\le k\}\cap B_r^+} x_n^{p+1} \,\ud x\\
&\ge C r^{-n-p-1} (\ell-k)\int_{\{u\ge l\}\cap B_r^+}  x_n^{p+1}\,\ud x\int_{\{u\le k\}\cap B_r^+} x_n^{p+1}\,\ud x.
\end{align*}
Hence, the conclusion follows. 
\end{proof}

\section{H\"older estimates for a  nonlinear parabolic equation with weights}\label{sec:holder}

Let $G$ be a bounded function on $B_1^+$ satisfying  
\be\label{eq:G}
\frac{1}{\Lda} x_n \le G(x)\le \Lda x_n \quad \mbox{for }x\in B_1^+,
\ee
and $A(x)= (a_{ij}(x))_{n\times n}$ be a bounded matrix valued function in $B_1^+$ satisfying
\begin{equation}\label{eq:ellipticity}
\frac{1}{\Lda} |\xi|^2 \le \sum_{i,j=1}^n a_{ij}(x)\xi_i\xi_j \le \Lda |\xi|^2 \quad\forall\,\xi\in\R^n,\ x\in B_1^+, 
\end{equation} 
where $\Lda \ge 1$ is constant.  

For $p\ge 1$, we consider positive bounded solutions of 
\begin{equation}\label{eq:main3}
G^{p+1} \partial_t w^p=\mbox{div}(A G^2\nabla w) \quad\mbox{in }B_1^+ \times (-1, 0]
\end{equation} 
satisfying 
\be\label{eq:lo-up}
\overline m\le w\le \overline M \quad \mbox{in }B_1^+  \times (-1, 0]\ \  \mbox{for some }0<\overline m\le \overline M<\infty.
\ee 
The equation \eqref{eq:main3} will be understood in the sense of distribution, and we are interested in the a priori H\"older estimates of its solutions that are Lipschitz continuous in $\overline B_1^+\times(-1,0]$. This Lipschitz continuity is assumed only for simplicity to avoid introducing more notations, and it is enough for our purpose. The H\"older estimate proved in the below clearly holds for properly defined weak solutions using approximations. 
Since the  operator $\mbox{div}(A G^2\nabla\,\cdot )$ is very degenerate near the boundary $\pa' B_1^+$, no boundary condition on $\pa' B_1^+$ should be  imposed; see Keldys \cite{Keldys},  Oleinik-Radkevic \cite{OR} and the more recent paper Wang-Wang-Yin-Zhou \cite{WWYZ}. 

The main result of this section is as follows. 

\begin{thm}\label{thm:holdernearboundary} 
Let $p\ge1$.  Suppose \eqref{eq:G} and \eqref{eq:ellipticity} hold. Suppose $w$ is Lipschitz continuous in $\overline B_1^+\times(-1,0]$, satisfies \eqref{eq:lo-up}, and is a solution of \eqref{eq:main3} in the sense of distribution. Then there exist $\gamma>0$ and $C>0$, both of which depend only on $n$, $p$, $\Lambda$, $\overline m$ and $\overline M$, such that
\[
|w(x,t)-w(y,s)|\le C (|x-y|+|t-s|)^\gamma\quad\forall\,(x,t), (y,s)\in \overline B_{1/2}^+\times(-1/4,0].
\]
\end{thm}

The equation \eqref{eq:main3} is uniformly parabolic when $x$ stays away from $\{x_n=0\}$. Therefore, we will establish the H\"older estimates at the boundary, and then, use a scaling argument and the interior H\"older estimates to establish the global H\"older estimate in Theorem \ref{thm:holdernearboundary}.

Throughout this section, we assume all the assumptions in Theorem \ref{thm:holdernearboundary}. We first establish the improvement of oscillation at the boundary.  

Let us start with a Caccioppoli type inequality. 

\begin{lem}\label{lem:degiorgiclass}
Let $k\in [\overline m, \overline M]$ and $\eta$ be a smooth function supported in $B_R(x_0)\times(-1,1)$, where $ B_R(x_0)\subset B_1$. Then for every $-1<t_1\le t_2\le 0$, we have
\begin{equation}\label{eq:caccipoli1}
\begin{split}
&\sup_{t_1<t<t_2} \int_{B_R^+(x_0)} v^{2}\eta^2 G^{p+1} \,\ud x  +\int_{B_R^+(x_0)\times(t_1,t_2] } |\nabla (v\eta)|^2 G^2\,\ud x \ud t  \\ 
&\le \int_{B_R^+(x_0)} (v^{2}+Cv^3)\eta^2 G^{p+1} \,\ud x\Big|_{t_1}  +  C \int_{B_R^+(x_0)\times(t_1,t_2] } \Big (|\nabla \eta|^2 G^2+ |\pa_t\eta|\eta G^{p+1}\Big) v^{2}\,\ud x\ud t,
\end{split}
\end{equation} 
and
\begin{equation}\label{eq:caccipoli2}
\begin{split}
&\sup_{t_1<t<t_2} \int_{B_R^+(x_0)} (\tilde v^{2}-C \tilde v^{3})\eta^2 G^{p+1} \,\ud x  +\int_{B_R^+(x_0)\times(t_1,t_2] } |\nabla (\tilde v\eta)|^2 G^2\,\ud x \ud t  \\ 
&\le \int_{B_R^+(x_0)} \tilde v^{2}\eta^2 G^{p+1} \,\ud x\Big|_{t_1}   + C \int_{B_R^+(x_0)\times(t_1,t_2] } \Big (|\nabla \eta|^2 G^2+ |\pa_t\eta|\eta G^{p+1}\Big) \tilde v^{2}\,\ud x\ud t,
\end{split}
\end{equation} 
where $v=(w-k)^+$, $\tilde v=(w-k)^-$, and $C>0$ depends only on $n$, $p$, $\Lambda$, $\overline m$ and $\overline M$.
\end{lem}
\begin{proof}
By the assumptions of $w$ in Theorem \ref{thm:holdernearboundary}, we have
\begin{equation}\label{eq:integrationbyparts}
\begin{split}
\int_{B_1^+ \times (-1, s]} G^{p+1} \varphi  \partial_t w^p \,\ud x\ud t  = -\int_{B_1^+ \times (-1, s]}G^2 A \nabla w \cdot \nabla\varphi \,\ud x\ud t 
\end{split}
\end{equation}
for every $s\in(-1,0]$ and every test function $\varphi$ that is Lipschitz continuous in $\overline Q_1^+$,  and vanishes on $\pa'' B_1^+\times[-1,0]$. 

First, by using $v\eta^2$ as a test function, we have for the left-hand side in \eqref{eq:integrationbyparts} that
\begin{align*}
\int_{-1}^s \int_{B_1^+} G^{p+1}v\eta^2 \partial_t w^p 
&= p\int_{-1}^s \int_{B_1^+} G^{p+1}v\eta^2 w^{p-1}\partial_t v\\
&= p\int_{-1}^s \int_{B_1^+} G^{p+1}\eta^2[(v+k)^{p}-k(v+k)^{p-1}]\partial_t v\\
&= p\int_{-1}^s \int_{B_1^+} G^{p+1}\eta^2\partial_t \left(\frac{(v+k)^{p+1}}{p+1}-\frac{k(v+k)^p}{p}+\frac{k^{p+1}}{p(p+1)}\right).
%&=\frac{p}{2}  \int_{B_1^+} G^{p+1}\eta^2(x,s)  (w^{p-1}v^2-\frac{p-1}{p+1}v^{p+1})(x,s)\,\ud x\\
%&\quad -p\int_{-1}^s \int_{B_1^+} G^{p+1}\eta\eta_t (w^{p-1}v^2-\frac{p-1}{p+1}v^{p+1}).
\end{align*}
Let 
\[
\theta_1(v,k,p)= \frac{(v+k)^{p+1}}{p+1}-\frac{k(v+k)^p}{p}+\frac{k^{p+1}}{p(p+1)}. 
\]
It is easy to see that 
\be  \label{eq:nonlinear-term}
\frac{1}{2}k^{p-1}v^2\le\theta_1(v,k,p) \le \frac{1}{2}k^{p-1}v^2 + C_0 v^3
\ee
for some $C_0>0$ depending only on $\overline m,\overline M, p$.  Since $v+k\le M\le \overline M$ and $k\ge m\ge \overline m$, we then have

\begin{align*}
&\int_{-1}^s \int_{B_1^+} G^{p+1}v\eta^2 \partial_t w^p \,\ud x\ud t \\
&\ge  \frac{p}{2} k^{p-1}\int_{B_1^+} G^{p+1}\eta^2(x,s) v^2(x,s)\,\ud x-  \frac{p}{2} k^{p-1}\int_{B_1^+} G^{p+1}\eta^2(x,t_1) v^2(x,t_1)\,\ud x  \\
&\quad-  C \int_{B_1^+} G^{p+1}\eta^2(x,t_1) v^3(x,t_1)\,\ud x -C\int_{-1}^s \int_{B_1^+} G^{p+1}\eta|\eta_t| v^{2}\,\ud x\ud t.
\end{align*}
Secondly, we have the right-hand side of \eqref{eq:integrationbyparts} that
\begin{align*}
\int_{Q_1} (A\nabla v)\nabla (v\eta^2)G^2\ud x\ud t\ge \frac{1}{2\Lambda} \int_{Q_1} |\nabla (\eta v)|^2G^2\,\ud x\ud t- C\int_{Q_1} |\nabla \eta |^2v^2G^2\,\ud x\ud t. 
\end{align*}
Then \eqref{eq:caccipoli1} follows.

The proof of \eqref{eq:caccipoli2} is similar by using $-\tilde v\eta^2$ as a test function. We have
\begin{align*}
-\int_{-1}^s \int_{B_1^+} G^{p+1}\tilde v\eta^2 \partial_t w^p 
&= p\int_{-1}^s \int_{B_1^+} G^{p+1}\tilde v\eta^2 w^{p-1}\partial_t\tilde  v\\
&= p\int_{-1}^s \int_{B_1^+} G^{p+1}\eta^2[-(k-\tilde v)^{p}+k(k-\tilde v)^{p-1}]\partial_t \tilde v\\
&= p\int_{-1}^s \int_{B_1^+} G^{p+1}\eta^2\partial_t \left(\frac{(k-\tilde v)^{p+1}}{p+1}-\frac{k(k-\tilde  v)^p}{p}+\frac{k^{p+1}}{p(p+1)}\right).
%&=\frac{p}{2}  \int_{B_1^+} G^{p+1}\eta^2(x,s)  (w^{p-1}v^2-\frac{p-1}{p+1}v^{p+1})(x,s)\,\ud x\\
%&\quad -p\int_{-1}^s \int_{B_1^+} G^{p+1}\eta\eta_t (w^{p-1}v^2-\frac{p-1}{p+1}v^{p+1}).
\end{align*}
Since
\[
\frac{1}{2} k^{p-1}\tilde v^2-C\tilde v^3\le \frac{(k-\tilde v)^{p+1}}{p+1}-\frac{k(k-\tilde v)^p}{p}+\frac{k^{p+1}}{p(p+1)}\le \frac{1}{2} k^{p-1}\tilde v^2,
\]
 we have
 \begin{align*}
&-\int_{-1}^s \int_{B_1^+} G^{p+1}\tilde v\eta^2 \partial_t w^p \,\ud x\ud t \\
&\ge  \frac{p}{2} k^{p-1}\int_{B_1^+} G^{p+1}\eta^2(x,s) \tilde v^2(x,s)\,\ud x -  C \int_{B_1^+} G^{p+1}\eta^2(x,s)\tilde v^3(x,s)\,\ud x  \\
&\quad -  \frac{p}{2} k^{p-1}\int_{B_1^+} G^{p+1}\eta^2(x,t_1) \tilde v^2(x,t_1)\,\ud x  -C\int_{-1}^s \int_{B_1^+} G^{p+1}\eta|\eta_t| \tilde v^{2}\,\ud x\ud t.
\end{align*}
Hence, \eqref{eq:caccipoli2} follows in a similar way.
\end{proof}

For $x_0\in\partial \R^n_+$ and $R>0$, let 
$$Q_R(x_0,t_0):=B_R(x_0)  \times (t_0-R^{p+1}, t_0], \quad Q_R^+(x_0,t_0):=B_R^+(x_0)  \times (t_0-R^{p+1}, t_0].$$ We simply write them as $Q_R$ and $Q_R^+$ if $(x_0,t_0)=(0,0)$.

For $q\ge 1$, let  $$\ud \mu_{q}= G^q \,\ud x,\quad \ud \nu_{q}= G^q \,\ud x\ud t$$  and 
\[
|A|_{\mu_q}= \int_{A} G^q \,\ud x\ \ \mbox{for }A\subset B_1^+, \quad  |\tilde A|_{\nu_q}=\int_{\tilde A} G^q \,\ud x\ud t\ \ \mbox{for }\widetilde A\subset Q_1^+.
\]
It is easy to find that for  $\tilde A\subset Q_R^+$, 
\be \label{eq:connection}
\frac{|\tilde A|_{\nu_{2}}}{|Q_R^+|_{\nu_{2}}}\le  \frac{|\tilde A|_{\nu_{p+1}}^{\frac{2}{p+1}}|\tilde A|_{\nu_{0}}^{\frac{p-1}{p+1}}}{|Q_R^+|_{\nu_{2}}}\le C  \frac{|\tilde A|_{\nu_{p+1}}^{\frac{2}{p+1}}R^{\frac{(p-1)(n+p+1)}{p+1}}}{R^{n+p+3}} \le C \left( \frac{|\tilde A|_{\nu_{p+1}}}{|Q_R^+|_{\nu_{p+1}}} \right)^{\frac{2}{p+1}} ,
\ee
where $C=C(n,\Lambda,p) >0$ depends only on $n,p$ and $\Lambda$, since 
\[
 \frac{R^{\frac{(p-1)(n+p+1)}{p+1}}}{R^{n+p+3}}= \frac{1}{R^{\frac{2(n+2p+2)}{p+1}}}\le C(n,\Lambda,p) |Q_R^+|_{\nu_{p+1}}^{-\frac{2}{p+1}}.
\]

The following lemma shows that if $w$ is small on a large portion of $Q_R^+$ measured by $\nu_{p+1}$, then we can improve its oscillation in $Q_{R/2}^+$.

\begin{lem}\label{lem:smallonlargeset} 
Let $0<R<1$ and
\[
\overline m\le m\le \inf_{Q_R^+} w\le \sup_{Q_R^+} w\le M\le \overline M.
\]
There exists  $0<\gamma_0<1$ depending only on $n$, $p$, $\Lambda$, $\overline m$ and $\overline M$ such that for  every $\delta >0$, if 
\[
\frac{|\{(x,t)\in Q_R^+: w(x,t)>M-\delta\} |_{\nu_{p+1}}}{|Q_R^+|_{\nu_{p+1}}} \le \gamma_0,
\]
then
\[
w\le M-\frac{\delta}{2}\quad\mbox{in }Q_{R/2}^+.
\]
\end{lem}

\begin{proof}
Let
\[
r_j=\frac {R}2+\frac{R}{2^{j+1}},\quad k_j=M-\delta+ \frac{\delta}{2}(1-2^{-j}),\quad j=0,1,2,\cdots.
\]
Let $\eta_j $ be a smooth cut-off function satisfying 
\[
\mbox{supp}(\eta_j) \subset Q_{r_j}, \quad 0\le \eta_j \le 1, \quad \eta_j=1 \mbox{ in }Q_{r_{j+1}}^+, 
\]
\[
|\nabla \eta_j(x,t)|^2+|\pa_t \eta_j(x,t)| R^{p-1}  \le \frac{C(n)}{(r_j-r_{j+1})^2} \quad \mbox{in }Q_R^+. 
\]
By Lemma \ref{lem:degiorgiclass},  Lemma \ref{lem:weightedsobolev} and \eqref{eq:G}, we have
\[
 \Big(\int_{Q_R^+}  |\eta_j v|^{2\chi}x_n^2\,\ud x \ud t \Big)^{\frac{1}{\chi}} \le C \int_{Q_{R}^+} \Big (|\nabla \eta_j|^2 + |\pa_t\eta_j | x_n^{p-1}\Big) x_n^2 v^{2} \,\ud x\ud t,
\]
where $v=(w-k_j)^+$. Let $A(k,\rho)= \{(x,t)\in Q_{\rho}^+: w\ge k\}$  for $0<\rho\le R$. Then 
\[
 \Big(\int_{Q_R^+}  |\eta_j v|^{2\chi}x_n^2\,\ud x \ud t \Big)^{\frac{1}{\chi}} \ge (k_{j+1}- k_j)^2 |A(k_{j+1}, r_{j+1})|_{\nu_2}^{\frac{1}{\chi}},  
\]
\[
\int_{Q_{R}^+} \Big (|\nabla \eta_j|^2 + |\pa_t\eta_j | x_n^{p-1}\Big) x_n^2 v^{2} \,\ud x\ud t \le \frac{C}{(r_j-r_{j+1})^2} (M-k_j)^2 |A(k_{j}, r_{j})|_{\nu_2} . 
\]
It follows that 
\begin{align*}
\frac{|A(k_{j+1}, r_{j+1})|_{\nu_2}}{|Q_R^+|_{\nu_2}} & \le  C 2^{4(j+2)\chi } R^{(n+p+3)(\chi-1)-2\chi} \left( \frac{|A(k_{j}, r_{j})|_{\nu_2} }{|Q_R^+|_{\nu_2}} \right)^{\chi}\\&
= C 2^{4(j+2)\chi } \left( \frac{|A(k_{j}, r_{j})|_{\nu_2} }{|Q_R^+|_{\nu_2}} \right)^{\chi} . 
\end{align*}
Therefore, by choosing $\gamma_0$ small, it follows from  \eqref{eq:connection} that we can make $\frac{|A(k_{0}, r_{0})|_{\nu_2} }{|Q_R^+|_{\nu_2}}$ small enough, so that
\[
\lim_{j\to \infty}\frac{|A(k_{j+1}, r_{j+1})|_{\nu_2}}{|Q_R^+|_{\nu_2}}=0.
\]
Then the conclusion follows.
\end{proof}

Similarly, if $w$ is large on a large portion of $Q_R^+$ measured by $\nu_{p+1}$, then we can improve its oscillation in $Q_{R/2}^+$ as well.

\begin{lem}\label{lem:smallonlargeset2}
Let $0<R<1$ and
\[
\overline m\le m\le \inf_{Q_R^+} w\le \sup_{Q_R^+} w\le M\le \overline M.
\]
There exist  $0<\gamma_0<1$ and $0<\delta_0<1$, both of which depend only on $n$, $p$, $\Lambda$, $\overline m$ and $\overline M$, such that for  every $0<\delta \le \delta_0$, if 
\[
\frac{|\{(x,t)\in Q_R^+: w(x,t)<m+\delta\} |_{\nu_{p+1}}}{|Q_R^+|_{\nu_{p+1}}} \le \gamma_0,
\]
then
\[
w\ge m+\frac{\delta}{2}\quad\mbox{in }Q_{R/2}.
\]
\end{lem}
\begin{proof}
Let
\[
r_j=\frac {R}2+\frac{R}{2^{j+1}},\quad \widetilde k_j=m+\delta- \frac{\delta}{2}(1-2^{-j}),\quad j=0,1,2,\cdots.
\]
There exists $\delta_0>0$ depending only on $n$, $p$, $\Lambda$, $\overline m$ and $\overline M$  such that for $0<\delta<\delta_0$ and $\tilde v= (w-\widetilde k_j)^-$, there holds
\[
\tilde v^2-C\tilde v^3\ge \frac{1}{2} \tilde v^2,
\]
where the constant $C$ in the above is the one in \eqref{eq:caccipoli2}.

The left proof is then identical to that of Lemma \ref{lem:smallonlargeset}.
\end{proof}

Next, we estimate the decay of the distribution function of $w$.

\begin{lem}\label{lem:decay-1} Let $0<R<\frac12$, $0<a\le 1$, $-\frac12<t_0\le -aR^{p+1}$, $0<\sigma<1$, $\delta>0$, and $\overline M\ge M_a\ge \sup_{B_{2R}^+ \times [t_0, t_0+aR^{p+1}] } w$.  If 
\[
|\{x\in B_R^+: w(x,t)>M_a-\delta\}|_{\mu_{p+1}}\le (1-\sigma) |B_R^+|_{\mu_{p+1}} \quad \mbox{for any } t_0\le t\le  t_0+aR^{p+1} ,
\]
then 
\[
\frac{|\{(x,t)\in B_{R}^+ \times [t_0, t_0+aR^{p+1}] : w(x,t)>M_a-\frac{\delta}{2^\ell}\}|_{\nu_{p+1}}}{|B_{R}^+ \times [t_0, t_0+aR^{p+1}]|_{\nu_{p+1}}}\le \frac{C}{\sigma \sqrt{a \ell}}\quad\forall\,\ell\in\mathbb{Z}^+,
\]
where $C$ depends only on $n$, $p$, $\Lambda$, $\overline m$ and $\overline M$.
\end{lem} 

\begin{proof} 
Let 
\[
A(k,R;t)= B_R^+ \cap \{w(\cdot, t)>k\}, \quad A(k,R)= B_{R}^+ \times [t_0, t_0+aR^{p+1}]  \cap \{w>k\}
\]
and 
\[
k_j= M_a- \frac{\delta}{2^j}. 
\]
By Proposition \ref{prop:degiorgiisoperimetricelliptic}, we have 
\begin{align*}
&(k_{j+1}-k_j) |A(k_{j+1},R;t)|_{\mu_{p+1}} |B_R^+\setminus A(k_{j},R;t)|_{\mu_{p+1}}\\
& \le C R^{n+p+2} \left(\int_{A(k_{j},R;t) \setminus A(k_{j+1},R;t)} |\nabla w|^2 x_n^2\right)^{1/2} |A(k_{j},R;t) \setminus A(k_{j+1},R;t)|_{\mu_{2p}}^{1/2}\\
& \le C R^{n+p+2} \left(\int_{B_R^+} |\nabla  (w-k_j)^+|^2 x_n^2\right)^{1/2} |A(k_{j},R;t) \setminus A(k_{j+1},R;t)|_{\mu_{2p}}^{1/2}. 
\end{align*}
By the assumption, 
\[
|B_R^+\setminus A(k_{j},R;t)|_{\mu_{p+1}} \ge \sigma |B_R^+|_{\mu_{p+1}}= C(n,p) \sigma R^{n+p+1}. 
\]
We also have 
\[
|A(k_{j},R;t) \setminus A(k_{j+1},R;t)|_{\mu_{2p}}^{1/2} \le R^{\frac{p-1}{2}} |A(k_{j},R;t) \setminus A(k_{j+1},R;t)|_{\mu_{p+1}}^{1/2} . 
\]
Integrating in the time variable, and using H\"older's inequality, we have
\begin{align*}
&\int_{t_0}^{t_0+a R^{p+1}} |A(k_{j+1},R;t)|_{\mu_{p+1}} \,\ud t\\&
\le \frac{C2^{j+1}}{\delta \sigma} R^{\frac{p+1}{2}} |A(k_{j},R) \setminus A(k_{j+1},R)|_{\nu_{p+1}}^{1/2}\left(\int_{B_R^+\times [t_0, t_0+aR^{p+1}]  } |\nabla (w-k_j)^+|^2 x_n^2\,\ud x \ud t\right)^{1/2} .
\end{align*}
Let $\eta(x) $ be a smooth cut-off function satisfying 
\begin{equation}\label{eq:testfunction}
\begin{split}
&\mbox{supp}(\eta) \subset B_{2R}, \quad 0\le \eta \le 1, \quad \eta=1 \mbox{ in }B_{R}, \quad |\nabla \eta(x)|^2  \le \frac{C(n)}{R^2} \quad \mbox{in }B_{2R}. 
\end{split}
\end{equation}
It follows from \eqref{eq:caccipoli1}  that
\begin{align*}
&\int_{t_0}^{t_0+a R^{p+1}}\int_{B_R^+  } |\nabla (w-k_j)^+|^2 x_n^2\,\ud x \ud t \\&
\le C \left(\int_{B_{2R}^+} |(w-k_j)^+(t_0)|^2 x_n^{p+1} \,\ud x+\frac{1}{R^2}\int_{t_0}^{t_0+a R^{p+1}}\int_{B_{2R}^+  } x_n^2 |(w-k_j)^+|^{2} \,\ud x\ud t \right )\\&
\le \frac{C \delta^2}{ 4^j} R^{n+p+1}. 
\end{align*}
Hence, 
\[
 |A(k_{j+1},R)|_{\nu_{p+1}} \le \frac{C}{\sigma} R^{\frac{n+2p+2}{2}} |A(k_{j},R) \setminus A(k_{j+1},R)|_{\nu_{p+1}}^{1/2}
\]
or 
\[
 |A(k_{j+1},R)|_{\nu_{p+1}}^2 \le \frac{C}{\sigma^2} R^{n+2p+2} |A(k_{j},R) \setminus A(k_{j+1},R)|_{\nu_{p+1}}. 
\]
Taking a summation, we have 
\begin{align*}
\ell |A(k_{\ell},R)|_{\nu_{p+1}}^2 &\le \sum_{j=0}^{\ell-1}  |A(k_{j+1},R)|_{\nu_{p+1}}^2 \le  \frac{C}{\sigma^2} R^{n+2p+2} |B_{R}^+ \times [t_0, t_0+aR^{p+1}]|_{\nu_{p+1}} \\& \le  \frac{C}{a\sigma^2}  |B_{R}^+ \times [t_0, t_0+aR^{p+1}]|_{\nu_{p+1}}^2. 
\end{align*}
The lemma follows. 
\end{proof}

Similarly,

\begin{lem}\label{lem:decay-1'} Let $0<R<\frac12$,  $0<a\le 1$, $-\frac12<t_0\le -aR^{p+1}$,  $0<\sigma<1$, $\delta>0$ and $\overline m\le m_a\le \inf_{B_{2R}^+ \times [t_0, t_0+aR^{p+1}] } w$. There exists $\delta_0>0$ depending only on $n$, $p$, $\Lambda$, $\overline m$ and $\overline M$ such that for  every $0<\delta \le \delta_0$, if
\[
|\{x\in B_R^+: w(x,t)<m_a+\delta\}|_{\mu_{p+1}}\le (1-\sigma) |B_R^+|_{\mu_{p+1}} \quad \mbox{for any } t_0\le t\le  t_0+aR^{p+1} ,
\]
then 
\[
\frac{|\{(x,t)\in B_{R}^+ \times [t_0, t_0+aR^{p+1}] : w(x,t)<m_a+\frac{\delta}{2^\ell}\}|_{\nu_{p+1}}}{|B_{R}^+ \times [t_0, t_0+aR^{p+1}]|_{\nu_{p+1}}}\le \frac{C}{\sigma \sqrt{a \ell}} \quad\forall\,\ell\in\mathbb{Z}^+,
\]
where $C$ depends only on $n$, $p$, $\Lambda$, $\overline m$ and $\overline M$.
\end{lem}  
\begin{proof} 
Let 
\[
\widetilde A(k,R;t)= B_R^+ \cap \{w(\cdot, t)<k\}, \quad \widetilde A(k,R)= B_{R}^+ \times [t_0, t_0+aR^{p+1}]  \cap \{w<k\}
\]
and 
\[
\widetilde k_j= m_a+ \frac{\delta}{2^j}. 
\]
By Proposition \ref{prop:degiorgiisoperimetricelliptic}, we have 
\begin{align*}
&(\widetilde k_{j}-\widetilde k_{j+1}) |\widetilde A(\widetilde k_{j+1},R;t)|_{\mu_{p+1}} |B_R^+\setminus \widetilde A(\widetilde k_{j},R;t)|_{\mu_{p+1}}\\
& \le C R^{n+p+2} \left(\int_{\widetilde A(\widetilde k_{j},R;t) \setminus \widetilde A(\widetilde k_{j+1},R;t)} |\nabla w|^2 x_n^2\right)^{1/2} |\widetilde A(\widetilde k_{j},R;t) \setminus \widetilde A(\widetilde k_{j+1},R;t)|_{\mu_{2p}}^{1/2}\\
& \le C R^{n+p+2} \left(\int_{B_R^+} |\nabla  (w-\widetilde k_j)^-|^2 x_n^2\right)^{1/2} |\widetilde A(\widetilde k_{j},R;t) \setminus \widetilde A(\widetilde k_{j+1},R;t)|_{\mu_{2p}}^{1/2}. 
\end{align*}
By the assumption, 
\[
|B_R^+\setminus \widetilde A(\widetilde k_{j},R;t)|_{\mu_{p+1}} \ge \sigma |B_R^+|_{\mu_{p+1}}= C(n,p) \sigma R^{n+p+1}. 
\]
We also have 
\[
|\widetilde A(\widetilde k_{j},R;t) \setminus \widetilde A(\widetilde k_{j+1},R;t)|_{\mu_{2p}}^{1/2} \le R^{\frac{p-1}{2}} |\widetilde A(\widetilde k_{j},R;t) \setminus \widetilde A(\widetilde k_{j+1},R;t)|_{\mu_{p+1}}^{1/2} . 
\]
Integrating in the time variable, and using H\"older's inequality, we have
\begin{align*}
&\int_{t_0}^{t_0+a R^{p+1}} |\widetilde A(\widetilde k_{j+1},R;t)|_{\mu_{p+1}} \,\ud t\\&
\le \frac{C2^{j+1}}{\delta \sigma} R^{\frac{p+1}{2}} |\widetilde A(\widetilde k_{j},R) \setminus \widetilde A(\widetilde k_{j+1},R)|_{\nu_{p+1}}^{1/2}\left(\int_{B_R^+\times [t_0, t_0+aR^{p+1}]  } |\nabla (w-\widetilde k_j)^-|^2 x_n^2\,\ud x \ud t\right)^{1/2} .
\end{align*}

There exists $\delta_0>0$ depending only on $n$, $p$, $\Lambda$, $\overline M$ and $\overline m$ such that for $0<\delta<\delta_0$ and $\tilde v= (w-\widetilde k_j)^-$, there holds
\[
\tilde v^2-C\tilde v^3\ge \frac{1}{2} \tilde v^2,
\]
where the constant $C$ in the above is the one in \eqref{eq:caccipoli2}. By choosing $\eta$ to be the one in \eqref{eq:testfunction}, it follows from \eqref{eq:caccipoli2}  that
\begin{align*}
&\int_{t_0}^{t_0+a R^{p+1}}\int_{B_R^+  } |\nabla (w-\widetilde k_j)^-|^2 x_n^2\,\ud x \ud t \\&
\le C\left( \int_{B_{2R}^+} |(w-\widetilde k_j)^-(t_0)|^2 x_n^{p+1} \,\ud x+\frac{1}{R^2}\int_{t_0}^{t_0+a R^{p+1}}\int_{B_{2R}^+  } x_n^2 |(w-\widetilde k_j)^-|^{2} \,\ud x\ud t\right) \\
&\le \frac{C \delta^2}{ 4^j} R^{n+p+1}. 
\end{align*}
Hence, 
\[
 |\widetilde A(\widetilde k_{j+1},R)|_{\nu_{p+1}} \le \frac{C}{\sigma} R^{\frac{n+2p+2}{2}} |\widetilde A(\widetilde k_{j},R) \setminus \widetilde A(\widetilde k_{j+1},R)|_{\nu_{p+1}}^{1/2}
\]
or 
\[
 |\widetilde A(\widetilde k_{j+1},R)|_{\nu_{p+1}}^2 \le \frac{C}{\sigma^2} R^{n+2p+2} |\widetilde A(\widetilde k_{j},R) \setminus \widetilde A(\widetilde k_{j+1},R)|_{\nu_{p+1}}. 
\]
Taking a summation, we have 
\begin{align*}
\ell |\widetilde A(\widetilde k_{\ell},R;t)|_{\nu_{p+1}}^2 &\le \sum_{j=0}^{\ell-1}  |\widetilde A(\widetilde k_{j+1},R;t)|_{\nu_{p+1}}^2 \le  \frac{C}{\sigma^2} R^{n+2p+2} |B_{R}^+ \times [t_0, t_0+aR^{p+1}]|_{\nu_{p+1}} \\& \le  \frac{C}{a\sigma^2}  |B_{R}^+ \times [t_0, t_0+aR^{p+1}]|_{\nu_{p+1}}^2. 
\end{align*}
The lemma follows. 
\end{proof}

The next two lemmas provide estimates on the distribution function of $w$ at each time slice based on the starting time.

\begin{lem}\label{lem:decay-2} Let $0<R<\frac12$, $-\frac12<t_0\le -R^{p+1}$, $\overline M\ge M_1\ge \sup_{B_{2R}^+ \times [t_0, t_0+R^{p+1}] } w$ and $0<\sigma<1$. There exist constants $\delta_0>0$ and $s_0>1$ depending only on $n$, $p$, $\Lambda$, $\overline m$, $\overline M$ and $\sigma$  such that if $0<\delta<\delta_0$ and 
\[
|\{x\in B_R^+: w(x,t_0)>M_1-\delta\}|_{\mu_{p+1}}\le (1-\sigma) |B_R^+|_{\mu_{p+1}}, 
\]
then
\[
|\{x\in B_R^+: w(x,t)>M_1-\frac{\delta}{2^{s_0}}\}|_{\mu_{p+1}}\le (1-\frac{\sigma}{2}) |B_R^+|_{\mu_{p+1}} \quad \mbox{for every } t_0\le t\le  t_0+R^{p+1} . 
\] 
\end{lem}  

\begin{proof}  Let $\eta$ be a cut-off function supported in $B_R$ and $\eta=1$ in $B_{\beta R}$, where $0<\beta<1$ to be fixed.  Let $0<a\le 1$  and 
\[
A^a(k,R)= \{B_R^+ \times[t_0, t_0+a R^{p+1}] \}\cap \{w>k\}.
\]
Let $k_1>1$.  By Lemma \ref{lem:degiorgiclass},  we have
\begin{align*}
&\sup_{t_0<t< t_0+a R^{p+1}} \int_{B_R^+} v^{2}\eta^2 G^{p+1} \,\ud x  \\&
\le \int_{B_R^+} (v^{2}+Cv^3)\eta^2 G^{p+1} \,\ud x\Big|_{t_0}  +  C \int_{B_R^+\times[t_0,t_0+a R^{p+1} ] } |\nabla \eta|^2 G^2 v^{2}\,\ud x\ud t,
\end{align*} 
where $v=(w- (M_1 -\delta))^+$. Note that 
\begin{align*}
\int_{B_R^+} v^{2}\eta^2 G^{p+1} \,\ud x \Big|_t &\ge \delta^2(1-2^{-k_1})^2  |B_{\beta R} \cap \{w(x,t)> M_1- \delta 2^{-k_1}\}|_{\mu_{p+1}},\\
\int_{B_R^+} (v^{2}+Cv^3)\eta^2 G^{p+1} \,\ud x\Big|_{t_0}  &\le \delta^2 (1+C \delta)|\{x\in B_R^+: w(x,t_0)>M_1-\delta\}|_{\mu_{p+1}} \\&\le  \delta^2 (1+C \delta) (1-\sigma) |B_R^+|_{\mu_{p+1}},
\end{align*}
and 
\begin{align*}
 \int_{B_R^+\times[t_0,t_0+a R^{p+1} ] } |\nabla \eta|^2 G^2 v^{2}\,\ud x\ud t &\le \delta^2 \frac{C}{(1-\beta)^2R^2} |A^a(M_1- \delta ,R)|_{\nu_2}\\&
 \le \delta^2|B_R^+|_{\mu_{p+1}} \frac{C }{(1-\beta)^2} \frac{|A^a(M_1- \delta ,R)|_{\nu_2} }{|Q_R|_{\nu_2}}.
\end{align*}
It follows that for all $t\in [t_0, t_0+a R^{p+1}]$, 
\begin{align*}
& |B_{\beta R}^+ \cap \{w(x,t)> M_1- \delta 2^{-k_1}\}|_{\mu_{p+1}} \\&\le |B_R^+|_{\mu_{p+1}}  \left( \frac{(1+C \delta) (1-\sigma)}{(1-2^{-k_1})^2 }+\frac{C }{(1-\beta)^2} \frac{|A^a(M_1- \delta ,R)|_{\nu_2} }{|Q_R|_{\nu_2}} \right).
\end{align*}
Hence,
\begin{align*}
& |B_{R}^+ \cap \{w(x,t)> M_1- \delta 2^{-k_1}\}|_{\mu_{p+1}} \\&\le |B_R^+|_{\mu_{p+1}}  \left( C(1-\beta)+\frac{(1+C \delta) (1-\sigma)}{(1-2^{-k_1})^2 }+\frac{C }{(1-\beta)^2  } \frac{|A^a(M_1- \delta ,R)|_{\nu_2} }{|Q_R|_{\nu_2}} \right).
\end{align*}
By choosing $\beta$ such that
\[
(1-\beta)^3=\frac{|A^a(M_1- \delta ,R)|_{\nu_2} }{|Q_R|_{\nu_2}}, 
\] 
we have
\begin{align}\label{eq:smallinitiallater}
& |B_{R}^+ \cap \{w(x,t)> M_1- \delta 2^{-k_1}\}|_{\mu_{p+1}} \nonumber\\
&\le |B_R^+|_{\mu_{p+1}}  \left(\frac{(1+C \delta) (1-\sigma)}{(1-2^{-k_1})^2 }+C \Big(\frac{|A^a(M_1- \delta ,R)|_{\nu_2} }{|Q_R|_{\nu_2}}\Big)^{\frac13} \right),
\end{align}
where $C>0$ depends only on $n,p,\Lambda,\overline m$ and $\overline M$. 
Since
\[
\frac{|A^a(M_1- \delta ,R)|_{\nu_2} }{|Q_R|_{\nu_2}}\le a,
\]
we can choose $a$ small such that
\[
Ca^{1/3}\le\frac{\sigma}{8}.
\] 
Now we fix such an $a$. We choose $a$ slightly smaller if necessary to make $a^{-1}$ to be an integer. Let $N=a^{-1}$ and denote
 \[
 t_j=t_0+jaR^{p+1}\quad j=1,2,\cdots,N.
 \]
 We will inductively prove that there exist $s_1<s_2<\cdots<s_N$ such that
 \begin{equation}\label{eq:densitypropa}
\sup_{t_{j-1}\le t\le t_j}|B_{ R}^+ \cap \{w(x,t)> M_1- \delta 2^{-s_j}\}|_{\mu_{p+1}}  \le \left(1-\sigma+\frac{j}{4N} \sigma\right)  |B_R^+|_{\mu_{p+1}},
\end{equation}
where all the $s_j$ depend only on  $n,p,\Lambda,\overline m, \overline M$ and $\sigma$, from which the conclusion of this lemma follow.

Let us consider $j=1$ first.   

There exist $\delta_0$ small and $k_0$ large, both of which depends only on $n$, $p$, $\Lambda$, $\overline m$, $\overline M$ and $\sigma$,  such that  for all $\delta\in(0,\delta_0]$ and all $k_1\ge k_0$, we have
\[
\frac{(1+C \delta) (1-\sigma)}{(1-2^{-k_1})^2 }\le 1-\sigma+\frac{\sigma}{8N}.
\]
Then by \eqref{eq:smallinitiallater}, 
\[
|B_{ R}^+ \cap \{w(x,t)> M_1- \delta 2^{-k_1}\}|_{\mu_{p+1}}  \le \left(1-\frac{3}{4} \sigma\right)  |B_R^+|_{\mu_{p+1}} 
\]
 for all $t\in [t_0, t_1]$.  Applying Lemma \ref{lem:decay-1} and \eqref{eq:connection}, for every $k_2>k_1$, we have
\[
 \frac{|A^a(M_1- \delta 2^{-k_2},R)|_{\nu_2} }{|Q_R|_{\nu_2}}\le  C\left(\frac{|A^a(M_1- \delta 2^{-k_2},R)|_{\nu_{p+1}} }{|Q_R|_{\nu_{p+1}}}\right)^{\frac{2}{p+1}} \le C\left(\frac{\sqrt{a}}{\sigma\sqrt{k_2-k_1}}\right)^{\frac{2}{p+1}}.
\]
Hence, we can choose $k_2$ large enough such hat
\[
C\left( \frac{|A^a(M_1- \delta 2^{-k_2},R)|_{\nu_2} }{|Q_R|_{\nu_2}}\right)^{\frac 13}\le \frac{\sigma}{8N}.
\]
Let $k_1=k_0$ and $s_1=k_1+k_2$. By replacing $\delta$ by $\delta 2^{-k_2}$ in \eqref{eq:smallinitiallater}, it follows that
\[
\sup_{t_{0}\le t\le t_1}|B_{ R}^+ \cap \{w(x,t)> M_1- \delta 2^{-s_1}\}|_{\mu_{p+1}}  \le \left(1-\sigma+\frac{1}{4N} \sigma\right)  |B_R^+|_{\mu_{p+1}}.
\]
This prove \eqref{eq:densitypropa} for $j=1$. The proof for $j=2,3,\cdots,N$ is similar, by considering the starting time as $t_{j-1}$. We omit the details.
\end{proof}

 Similarly,

\begin{lem}\label{lem:decay-2'} Let  $0<R<\frac12$, $-\frac12<t_0\le -R^{p+1}$, $\overline m\le m_1\le \inf_{B_{2R}^+ \times [t_0, t_0+R^{p+1}] } w$ and $0<\sigma<1$. There exist constants $\delta_0>0$ and $k_0>1$ depending only on  $n$, $p$, $\Lambda$, $\overline m$, $\overline M$ and $\sigma$  such that if $0<\delta<\delta_0$ and 
\[
|\{x\in B_R^+: w(x,t_0)<m_1+\delta\}|_{\mu_{p+1}}\le (1-\sigma) |B_R^+|_{\mu_{p+1}}, 
\]
then
\[
|\{x\in B_R^+: w(x,t)<m_1+\frac{\delta}{2^{k_0}}\}|_{\mu_{p+1}}\le (1-\frac{\sigma}{2}) |B_R^+|_{\mu_{p+1}} \quad \mbox{for any } t_0\le t\le  t_0+R^{p+1} . 
\] 
\end{lem}

\begin{proof}  Let $\eta$ be a cut-off function supported in $B_R$ and $\eta=1$ in $B_{\beta R}$, where $0<\beta<1$ to be fixed.  Let $0<a\le 1$  and 
\[
\widetilde A^a(k,R)= \{B_R^+ \times[t_0, t_0+a R^{p+1}] \}\cap \{w<k\}.
\]
Let $k_1>1$.  By Lemma \ref{lem:degiorgiclass},  we have
\begin{align*}
&\sup_{t_0<t< t_0+a R^{p+1}} \int_{B_R^+} (\tilde v^{2}-C\tilde v^3)\eta^2 G^{p+1} \,\ud x  \\
&\quad\le \int_{B_R^+}\tilde v^{2}\eta^2 G^{p+1} \,\ud x\Big|_{t_0}  +  C \int_{B_R^+\times[t_0,t_0+a R^{p+1} ] } |\nabla \eta|^2 G^2 \tilde v^{2}\,\ud x\ud t,
\end{align*} 
where $\tilde v=(w- (m_1 +\delta))^-$. Choose $\delta_0$ small such that $1-C\delta_0>1/2$. Note that 
\begin{align*}
\int_{B_R^+} (\tilde v^{2}-C\tilde v^3)\eta^2 G^{p+1} \,\ud x \Big|_t &\ge (1-C\delta)\delta^2(1-2^{-k_1})^2  |B_{\beta R} \cap \{w(x,t)< m_1+ \frac{\delta}{2^{k_1}}\}|_{\mu_{p+1}}, \\
\int_{B_R^+} \tilde v^{2}\eta^2 G^{p+1} \,\ud x\Big|_{t_0}  &\le \delta^2 |\{x\in B_R^+: w(x,t_0)<m_1+\delta\}|_{\mu_{p+1}} \\&\le  \delta^2  (1-\sigma) |B_R^+|_{\mu_{p+1}},
\end{align*}
and 
\begin{align*}
 \int_{B_R^+\times[t_0,t_0+a R^{p+1} ] } |\nabla \eta|^2 G^2\tilde v^{2}\,\ud x\ud t &\le \delta^2 \frac{C}{(1-\beta)^2R^2} |\widetilde A^a(m_1+ \delta ,R)|_{\nu_2}\\&
 \le \delta^2|B_R^+|_{\mu_{p+1}} \frac{C }{(1-\beta)^2} \frac{|\widetilde A^a(m_1+ \delta ,R)|_{\nu_2} }{|Q_R|_{\nu_2}}.
\end{align*}
It follows that for all $t\in [t_0, t_0+a R^{p+1}]$, 
\begin{align*}
& |B_{\beta R}^+ \cap \{w(x,t)< m_1+ \delta 2^{-k_1}\}|_{\mu_{p+1}} \\
&\le |B_R^+|_{\mu_{p+1}}  \left( \frac{ (1-\sigma)}{(1-C \delta)(1-2^{-k_1})^2 }+\frac{C }{(1-\beta)^2} \frac{|\widetilde A^a(m_1+ \delta ,R)|_{\nu_2} }{|Q_R|_{\nu_2}} \right)\\
&\le |B_R^+|_{\mu_{p+1}}  \left( \frac{(1+C \delta) (1-\sigma)}{(1-2^{-k_1})^2 }+\frac{C }{(1-\beta)^2} \frac{|\widetilde A^a(m_1+ \delta ,R)|_{\nu_2} }{|Q_R|_{\nu_2}} \right).
\end{align*}
Hence,
\begin{align*}
& |B_{R}^+ \cap \{w(x,t)< m_1+ \delta 2^{-k_1}\}|_{\mu_{p+1}} \\&\le |B_R^+|_{\mu_{p+1}}  \left( C(1-\beta)+\frac{(1+C \delta) (1-\sigma)}{(1-2^{-k_1})^2 }+\frac{C }{(1-\beta)^2  } \frac{|\widetilde A^a(m_1+ \delta ,R)|_{\nu_2} }{|Q_R|_{\nu_2}} \right),
\end{align*}
where $C>0$ depends only on $n$, $p$, $\Lambda$, $\overline m$ and $\overline M$. By choosing $\beta$ such that
\[
(1-\beta)^3=\frac{|\widetilde A^a(m_1+ \delta ,R)|_{\nu_2} }{|Q_R|_{\nu_2}}, 
\] 
we have
\begin{align}\label{eq:smallinitiallater3}
& |B_{R}^+ \cap \{w(x,t)< m_1+ \delta 2^{-k_1}\}|_{\mu_{p+1}} \nonumber\\
&\le |B_R^+|_{\mu_{p+1}}  \left(\frac{(1+C \delta) (1-\sigma)}{(1-2^{-k_1})^2 }+C \Big(\frac{|\widetilde A^a(m_1+ \delta ,R)|_{\nu_2} }{|Q_R|_{\nu_2}}\Big)^{\frac13} \right).
\end{align}
With the help of Lemma  \ref{lem:decay-1'}, the left proof is almost identical to that of Lemma \ref{lem:decay-2}, and we omit it.
\end{proof}

Finally, combining all the above lemmas, we obtain the improvement of oscillation of $w$ at the boundary.

\begin{lem}\label{lem:decay-3} Let $0<R<\frac12$, $\overline M\ge M\ge \sup_{B_{2R}^+ \times [-R^{p+1},0] } w$ and $0<\sigma<1$. There exist constants $\delta_0>0$ and $k_0>1$ depending only on $n$, $p$, $\Lambda$, $\overline m$, $\overline M$ and $\sigma$  such that if $0<\delta<\delta_0$ and 
\[
|\{x\in B_R^+: w(x,-R^{p+1})>M-\delta\}|_{\mu_{p+1}}\le (1-\sigma) |B_R^+|_{\mu_{p+1}}, 
\]
then 
\[
\sup_{Q_{R/2}^+} w\le M-\frac{\delta}{2^{k_0}}.  
\]

\end{lem}

\begin{proof} It follows from Lemma \ref{lem:decay-2},  Lemma \ref{lem:decay-1} with $a=1$ and Lemma \ref{lem:smallonlargeset}.  
\end{proof}

\begin{lem}\label{lem:decay-3'} Let $0<R<\frac12$, $\overline m\le m\le \inf_{B_{2R}^+ \times [-R^{p+1},0] } w$ and $0<\sigma<1$. There exist constants $\delta_0>0$ and $k_0>1$ depending only on  $n$, $p$, $\Lambda$, $\overline m$, $\overline M$ and $\sigma$ such that if $0<\delta<\delta_0$ and 
\[
|\{x\in B_R^+: w(x,-R^{p+1})<m+\delta\}|_{\mu_{p+1}}\le (1-\sigma) |B_R^+|_{\mu_{p+1}}, 
\]
then 
\[
\inf_{Q_{R/2}^+} w\ge m+\frac{\delta}{2^{k_0}}.  
\]

\end{lem}

\begin{proof} It follows from Lemma \ref{lem:decay-2'},  Lemma \ref{lem:decay-1'} with $a=1$ and Lemma \ref{lem:smallonlargeset2}.  
\end{proof}

Then the H\"older regularity at the boundary follows in a standard way.

\begin{thm}\label{thm:holderatboundary}
Let $p\ge1$.  Suppose \eqref{eq:G} and \eqref{eq:ellipticity} hold. Suppose $w$ is Lipschitz continuous in $\overline B_1^+\times(-1,0]$, satisfies \eqref{eq:lo-up}, and is a solution of \eqref{eq:main3} in the sense of distribution. 
Then there exist $\alpha>0$ and $C>0$, both of which depend only on $n$, $p$, $\Lambda$, $\overline m$ and $\overline M$, such that for every $\bar x\in\pa' B_{1/2}^+$ and $\bar t\in (-1/4,0)$, there holds
\[
|w(x,t)-w(\bar x,\bar t)|\le C (|x-\bar x|+|t-\bar t|^{\frac{1}{p+1}})^\alpha\quad\forall\,(x,t)\in B_{1/2}^+(\bar x)\times(-1/4+\bar t,\bar t].
\]
\end{thm}
\begin{proof}
Without loss of generality, we suppose $(\bar x,\bar t)=(0,0)$. Choose $\ell_0\ge 2$ large so that 
\[
\frac{\overline M-\overline m}{\ell_0}<\delta_0,
\]
where $\delta_0$ is smaller one between those in Lemma \ref{lem:decay-3} with $\sigma=\frac12$ or Lemma \ref{lem:decay-3'} with $\sigma=\frac12$.

For $R\in(0,1/2]$, let
\[
\mu(R)=\sup_{Q_R^+}w, \quad \widetilde\mu(R)=\inf_{Q_R^+}w, \quad\omega(R)=\mu(R)-\widetilde\mu(R).
\]
Then one of the following two inequalities must hold:
\begin{align}
\left|\left\{x\in B_{R/2}^+: w\Big(x,-(R/2)^{p+1}\Big)>\mu(R)- \frac{1}{\ell_0}\omega(R)\right\}\right|_{\mu_{p+1}}\le \frac12 |B_{R/2}^+|_{\mu_{p+1}}, \label{eq:dichonomy1}\\
\left|\left\{x\in B_{R/2}^+: w\Big(x,-(R/2)^{p+1}\Big)<\widetilde\mu(R)+ \frac{1}{\ell_0}\omega(R)\right\}\right|_{\mu_{p+1}}\le \frac12 |B_{R/2}^+|_{\mu_{p+1}}.\label{eq:dichonomy2}
\end{align}
If \eqref{eq:dichonomy1} holds, then by Lemma \ref{lem:decay-3}, there exists $s_0>1$ such that
\[
\mu(R/4)\le \mu(R)-\frac{\omega(R)}{\ell_02^{s_0}}.
\]
If \eqref{eq:dichonomy2} holds, then by Lemma \ref{lem:decay-3'}, there exists $s_0>1$ such that
\[
\widetilde\mu(R/4)\ge \widetilde\mu(R)+\frac{\omega(R)}{\ell_02^{s_0}}.
\]
In any case, we obtain
\[
\omega(R/4)\le \left(1-\frac{1}{\ell_02^{s_0}}\right)\omega(R).
\]
This implies that there exist two positive constants $\alpha$ and $C$, both of which depend only on $n$, $p$, $\Lambda$, $\overline m$ and $\overline M$, such that
\[
\omega(R)\le CR^{\alpha}\quad\,\forall R\in(0,1/2],
\]
from which the conclusion follows.
\end{proof}

Once we have the H\"older estimates at the boundary, we can use a scaling argument and the interior H\"older estimates to obtain the H\"older estimates up to $\{x_n=0\}$.
\begin{proof}[Proof of Theorem \ref{thm:holdernearboundary}]
For $R>0$, let 
\begin{equation}\label{eq:scalingfunctions}
\begin{split}
w_R(y,s)=w(Ry,R^{p+1}s),\ A_R(x)=A(Rx),\ G_R(x)&=R^{-1} G(Rx).
\end{split}
\end{equation}
Then $w_R$ satisfies
\begin{equation}\label{eq:rescaledequation}
G_R^{p+1} \partial_t w_R^p=\mbox{div}(A_R G_R^2\nabla w_R)  \quad\mbox{in }Q_{1/R}^+ .
\end{equation}
This equation is uniformly parabolic in $B_{1/2}(e_n)\times (-1,0]$, where $e_n=(0,\cdots,0,1)\in\R^n$.

For any $\bar x=(0,\bar x_n )\in B_{1/2}^+$, we let $R:=\bar x_n>0$ and $w_R,A_R,G_R$ as in \eqref{eq:scalingfunctions}. Applying the interior H\"older estimates for uniformly parabolic equations to \eqref{eq:rescaledequation} in $B_{1/2}(e_n)\times (-1,0]$, there exist $C>1$ and $0<\beta<1$, both depending only on $\Lambda,n,p,\overline m$ and $\overline M$, such that  
\begin{equation}\label{eq:holderafterscaling}
|w_R(e_n,0)-w_R(y,s)|\le C (|y-e_n|+\sqrt{|s|})^\beta
\end{equation}
for all $(y,s)$ such that $|y-e_n|+\sqrt{|s|}<1/2$.

Consider $t\in (-1/4,0]$. If $|t|\le R^{2p+2}$, then we have
\[
|w(\bar x,t)-w(\bar x, 0)|=|w_R(e_n,t/R^{p+1})-w_R(e_n,0)|\le C |t/R^{p+1}|^{\beta/2}\le C |t|^{\beta/4},
\]
where we used \eqref{eq:holderafterscaling} in the first inequality. If $|t|\ge R^{2p+2}$, then we have
\begin{align*}
|w(\bar x,t)-w(\bar x, 0)|& \le |w(\bar x,t)-w(0,t)|+|w(0,t)-w(0,0)|+|w(0,0)-w(\bar x, 0)|\\
& \le C(R^{\alpha}+|t|^{\frac{\alpha}{p+1}})\\
&\le C |t|^{\frac{\alpha}{2(p+1)}},
\end{align*}
where we used Theorem \ref{thm:holderatboundary} in the second inequality. This shows that $w$ is H\"older continuous in the time variable.

Consider $\tilde x=(\tilde x',\tilde x_n)\in B_{1/2}^+$ such that $\tilde x_n\le \bar x_n$. If $\tilde x\in B_{R^2}(\bar x)$, then we have 
\[
|w(\bar x,0)-w(\tilde x, 0)|=|w_R(e_n,0)-w_R(\tilde x/R,0)|\le C ||\tilde x-\bar x|/R|^{\beta}\le C |\tilde x-\bar x|^{\beta/2},
\]
where we used \eqref{eq:holderafterscaling} in the first inequality. If $\tilde x\not\in B_{R^2}(\bar x)$, since $\tilde x_n\le \bar x_n$, then we have 
\begin{align*}
&|w(\bar x,0)-w(\tilde x, 0)|\\
& \le |w(\bar x,0)-w(0,0)|+|w(0,0)-w(\tilde x',0,0)|+|w(\tilde x',0,0)-w(\tilde x, 0)|\\
& \le C(R^{\alpha}+|\tilde x'|^\alpha)\\
&\le C |\bar x-\tilde x|^{\frac{\alpha}{2}},
\end{align*}
where we used Theorem \ref{thm:holderatboundary} in the second inequality. This shows that $w$ is H\"older continuous in the spatial variables.

This finishes the proof of Theorem \ref{thm:holdernearboundary}.
\end{proof}

\section{Higher order estimates and proof of the main theorem}\label{sec:regularity}

Let $\om \subset \R^n$, $n\ge 1$, be a smooth bounded  domain. Let 
\[
\mathcal{L}=- \Delta -b,
\] where $b<\lda_1$ is a constant and $\lda_1$ is the first Dirichlet  eigenvalue of $-\Delta $ on $\om$. Recall that $d(x)=dist(x,\pa \om)$. Consider nonnegative solutions of the equation
\be\label{eq:FDE-1}
\begin{cases}
\pa_t v^p =-\mathcal{L} v \quad \mbox{in } \om \times (0,T],\\
v(x,t)=0\quad \mbox{on } \partial\om \times (0,T],
\end{cases}
\ee
where $1<p<\infty$ and $T>0$. Suppose that
\be \label{eq:pre-1}
\frac{1}{c_0}d(x) \le v(x,t)\le c_0 d(x)\quad \mbox{in } \om \times (0,T]
\ee
for some constant $c_0\ge 1$. Denote
\begin{align*}
C^{3, 2}(\overline \om\times [0,T])=\{&u\in C(\overline \om\times [0,T]): D_x u, D_x^2 u, D_x^3 u,\\ &\quad\quad\quad\quad\partial_t u, D_x \partial_t u, D_x^2 \partial_t u, \partial_t^2 u \in C(\overline \om\times [0,T])\}.
\end{align*}
Let $\alpha\in (0,1)$ and 
\begin{align*}
[u]_{\mathscr{C}^\alpha( \om \times (0,T])}&:=\sup_{\substack{(x,t),(y,t)\in  \om \times (0,T],\\ x\neq y}}\frac{|u(x,t)-u(y,t)|}{|x-y|^\alpha}+ \sup_{\substack{(x,t),(x,s)\in  \om \times (0,T], \\ t\neq s}}\frac{|u(x,t)-u(x,s)|}{|t-s|^\frac{\alpha}{p+1}}\\
&\quad+\sup_{\substack{(x,t),(x,s)\in  \om \times (0,T], \\ t\neq s}}d(x)^{\frac{(p-1)\alpha}{2}}\frac{|u(x,t)-u(x,s)|}{|t-s|^\frac{\alpha}{2}}.
\end{align*}
This weighted H\"older semi-norm was introduced in \cite{JX19} for establishing Schauder estimates of the linearized equation of \eqref{eq:FDE-1}. It incorporates the two different scalings of the equation \eqref{eq:FDE-1} when the scaling is centered  at either an interior point or a boundary point. 
Denote
$$\|u\|_{\mathscr{C}^{\al}(\overline \om\times [0,T])}= \|u\|_{L^\infty(\overline \om\times [0,T])}+[u]_{\mathscr{C}^{\al}(\overline \om\times [0,T])}.$$

\begin{thm}\label{thm:uttbound} 
Let $v \in C^{3,2}(\overline \om \times [0,T])$  be a positive solution of  \eqref{eq:FDE-1} satisfying \eqref{eq:pre-1}. 

(i). If $p$ is an integer, then $v\in C^\infty(\overline\om\times [T/2, T]))$, and 
\[
\|d^{-1} \pa_t^\ell v\|_{L^\infty(\om\times [T/2, T])} + \|D_x^k\partial_t^\ell v\|_{L^\infty(\om\times [T/2, T]) }\le C,
\] 
for every $k,\ell=0,1,2,\cdots$,  where $C>0$ depends only on  $n$, $\om$, $T$, $p$, $b$, $c_0$, $ \ell$ and $k$. 

\medskip

(ii). If $p$ is not an integer, then $\pa_t^\ell v(\cdot,t)\in C^{2+p}(\overline \om)$ for every $\ell=0,1,2,\cdots$ and all $t\in[T/2,T]$, and there holds
\[
\|d^{-1} \pa_t^\ell v\|_{L^\infty(\om\times [T/2, T])}+\sup_{t\in[T/2,T]}\|\partial_t^\ell v(\cdot,t)\|_{C^{2+p}(\overline\om) }\le C,
\]
where $C>0$ depends only on  $n$, $\om$, $T$, $p$, $b$, $c_0$ and $\ell$. 
\end{thm}
\begin{proof}
Step 1. We show that there exist $\alpha>0$ and $C>0$, both of which depend only on  $n$, $\om$, $T$, $p$, $b$ and $c_0$    such that
\be \label{eq:step1-a}
\left\|\frac{v}{d}\right\|_{\mathscr{C}^{\al}(\om \times [T/8, T])}\le C. 
\ee

Indeed, by \eqref{eq:pre-1} and the H\"older regularity theory of linear uniformly parabolic equations, we only need to show the H\"older estimation of $v/d$ near the lateral boundary.  We  pick  a point  $x_0\in \om$ staying far away from the boundary $\pa \om$ and let  $G$ be the Green's function centered at $x_0$, i.e., 
\[
\mathcal{L} G= \delta_{x_0} \quad \mbox{in }\om, \quad G=0 \quad \mbox{on }\pa \om. 
\] Then $\frac{1}{c_1}\le G/d\le c_1$ and $G/d$ is smooth in $\om_{\rho}:=\{x\in \om: d(x)<\rho\}$ for some constants $\rho< \frac{1}{2} d(x_0)$ and $c_1\ge 1$, both of which depend only on $\Omega, n$ and $b$.  Let
\[
w:=\frac{v}{G}.
\]
Then it is elementary to check that 
\[
G^{p+1} \pa_t w^p =\mathrm{div}(G^2 \nabla w) \quad \mbox{in }\om_{\rho} \times (0,T]. 
\]
By straightening out the boundary $\partial\Omega$, and using the assumption \eqref{eq:pre-1} and Theorem \ref{thm:holdernearboundary}, we have  
$$\| w\|_{\mathscr{C}^{\al}(\om_{\rho/2} \times [T/8,T])} \le C$$  
for some $C>0$ depending only  $n$, $\om$, $T$, $p$, $b$, $c_0$ and $c_1$. Therefore, \eqref{eq:step1-a} follows. 

\medskip

Step 2. Let $\eta(t)$ be a cutoff function satisfying $\eta=0$ for $t \le T/8$ and $\eta=1$ for $t\ge T/7$, and let $\tilde v=\eta v$. Then 
\[
pd^{p-1} \left(\frac{v}{d}\right)^{p-1}  \pa_t  \tilde v= -\mathcal{L} \tilde v+pd^{p-1} \left(\frac{v}{d}\right)^{p-1}  v\pa_t \eta \quad \mbox{in }\om \times (0,T].
\]
By \eqref{eq:step1-a} and the Schauder estimates (Theorem 2.19 of \cite{JX19}), we have 
\[
\|\pa_tv\|_{\mathscr{C}^{\al} (\overline\Omega \times [T/7,T])} +\|D_x v\|_{\mathscr{C}^{\al} (\overline\Omega \times [T/7,T])} + \sup_{t\in[T/7,T]} \|D_x^2 v(\cdot, t)\|_{C^{\beta} (\om)}\le C,
\]
where $\beta=\min\{\al, p-1\}$. 

\medskip

Step 3. We claim that 
\be \label{eq:AB-est}
\left\|\frac{\pa_t v}{ d}\right\|_{\mathscr{C}^{\al} (\overline\Omega \times [T/3,T])}\le C  \quad \mbox{in } \om\times [T/3, T]. 
\ee
Indeed, for $\lda\in (0,1]$ and arbitrarily  small positive number  $0<h<\frac{T}{100}$, we define 
\[
v^h_\lda(x,t)=\frac{v(x,t)-v(x,t-h)}{h^\lda}.
\]
By the equation of $v$, 
\be \label{eq:time-diff-quotient}
p d^{p-1} a \pa_t v_{\lda}^h =\mathcal{L} v_{\lda}^h +  d^{p-1}f \frac{v_{\lda}^h}{d}  \quad \mbox{in } \om \times (T/100,T],
\ee
where $a(x,t)=(\frac{v(x,t)}{d(x)})^{p-1}$
and
\begin{align*}
f(x,t)&= -d(x)^{2-p} \pa_t v(x,t-h) \int_0^1 \frac{\ud }{\ud \theta}[\theta v(x,t)+(1-\theta ) v(x,t-h)]^{p-1} \,\ud \theta     \\&
= -(p-1) \pa_t v(x,t-h)\int_0^1 \left[\theta \frac{v(x,t)}{d}+(1-\theta ) \frac{v(x,t-h)}{d}\right]^{p-2} \,\ud \theta.  
\end{align*}
By \eqref{eq:pre-1}, $\frac{1}{c_0}\le a\le c_0$. This together with \eqref{eq:step1-a} implies that  $$a, f\in\mathscr{C}^{\al} (\overline\Omega \times [T/6,T]). $$

Set ${\lda_m}=\frac{(m+1)\al}{p+1}$ if $m<\frac{p+1}{\al}-1$ and $\lda_m=1$ if $m\ge \frac{p+1}{\al}-1$. By Step 2 and Taylor expansion calculations, we have 
\be \label{eq:dq-1}
|\pa_t v_{\lda_0}^h |+\left|\frac{v^h_{\lda_0}} d\right|\le C \quad \mbox{in } \om\times [T/6, T]. 
\ee
Using \eqref{eq:dq-1} and applying elliptic estimates on each time slice of \eqref{eq:time-diff-quotient}, we have 
\[
\sup_{T/6\le t\le T}\|v_{\lda_0}^h (\cdot, t)\|_{C^{1,\gamma}(\om)} \le C(\gamma)
\]
for any $0<\gamma<1$. Combing with $|\pa_t v_{\lda_0}^h |\le C$ in \eqref{eq:dq-1}, it follows from a calculus lemma, Lemma 3.1 on page 78 in \cite{LSU} (cf. Lemma B.3 in \cite{JX19}), we have 
\[
|\nabla v_{\lda_0}^h(x,t) -\nabla v_{\lda_0}^h(y,s)| \le C(|x-y|^2+|t-s|)^{\frac{\gamma}{2}}, \quad \forall~ (x,t), (y,s)\in \om\times [T/6,T]. 
\]
By choosing $\gamma\ge \al$, we have
\[
\left\|\frac{v_{\lda_0}^h}{d}\right\|_{\mathscr{C}^{\al}(\overline\om\times [T/6,T])}\le C. 
\]
Applying the Schauder estimates (Theorem 2.19 of \cite{JX19}) to \eqref{eq:time-diff-quotient}, we then conclude that 
\be\label{eq:dq-2}
\|\pa_tv_{\lda_0}^h\|_{\mathscr{C}^{\al} (\overline\Omega \times [T/5,T])} +\|D_x v_{\lda_0}^h\|_{\mathscr{C}^{\al} (\overline\Omega \times [T/5,T])} + \sup_{t\in[T/5,T]} \|D_x^2 v_{\lda_0}^h(\cdot, t)\|_{C^{\beta} (\om)}\le C,
\ee
where $\beta=\min\{\al, p-1\}$ and $C>0$ is independent of $h$. It follows that 
\[
|\pa_t v_{\lda_1}^h |+\left|\frac{v^h_{\lda_1}} d\right|\le C \quad \mbox{in } \om\times [T/4, T].
\] 
Repeating the above argument in finite steps to obtain \eqref{eq:dq-2} with $\lda_0$ replaced by $1$ in $\overline\Omega \times [T/3,T]$, in particular,
\[
\|D_x \partial_t v\|_{\mathscr{C}^{\al} (\overline\Omega \times [T/3,T])}\le C, 
\]
from which \eqref{eq:AB-est} follows immediately. 

\medskip

Step 4. Once we have \eqref{eq:AB-est}, higher order estimates follow from the bootstrap arguments by differentiating the equation \eqref{eq:FDE-1} in the time variable, as in proof of Theorem 4.8 of \cite{JX19}.  Therefore, Theorem \ref{thm:uttbound} is proved.
\end{proof}

\begin{proof}[Proof of Theorem \ref{thm:main}]  Once we have the  \textit{a priori} estimates in Theorem \ref{thm:uttbound}, the left arguments of proving Theorem \ref{thm:main} is identical to that of Theorem 5.1 in \cite{JX19}, which  we sketch in the below.  

Let $u$ be a bounded nonnegative weak solution of \eqref{eq:main} and \eqref{eq:main-d}.    By  the continuity of $u$  (see Chen-DiBenedetto \cite{CDi}) and a translation in the time variable if needed, we can assume that $u_0=u(\cdot,0)\in C_0(\overline\Omega)$ and $u_0$ is not identical to zero. Let $T^*>0$ be the extinction time of $u$. It is proved in Section 5.10.1 of V\'azquez \cite{Vaz} that there exists $ q\in[p+ 1,\infty)$, which depends only on $n$ and $p$,  such that 
\[
\frac{\ud }{\ud t}\left( \int_{\om} u^{q}\,\ud x\right)^{\frac{p-1}{q}} \le -C(n,p,\om)<0. 
\] See (5.79) there.  Integrating the above differential  inequality from $t$ to $T^*$, we have  
\[
C(T^*-t)^{\frac{1}{p-1}}\le  \|u(\cdot, t)\|_{L^q(\om)}  \le \|u(\cdot, t)\|_{L^\infty(\om)}. 
\] By the maximum principle, $u(x, t)\le \max_{\overline \om} u_0$. Hence, for any $0<\delta<T^*/4$ we have 
\[
C\delta^{\frac{1}{p-1}}\le  \|u(\cdot, t)\|_{L^\infty(\om)} \le  \|u_0\|_{L^\infty(\om)} \quad \mbox{for }\delta<t<T^*-\delta,
\]
where $C>0$ depends only on $n,p,\om,\delta$ and $\|u_0\|_{L^\infty(\om)}$.  By the proof of Proposition 6.2 of DiBenedetto-Kwong-Vespri \cite{DKV},  
\be \label{eq:DKV-a}
\frac{1}{C}\le \frac{u}{d}\le C \quad \mbox{in }\om\times[\delta,T^*-\delta],
\ee 
where $C>0$ depending only on $n,p,\om,\delta$ and $\|u_0\|_{L^\infty(\om)}$. 

Define
\begin{align*} \label{eq:good-initial}
\mathcal{S}=\Big\{v\in C^\infty( \Omega) \cap C^{2+\alpha}_0(\overline \om):  \inf_{\om}\frac{v}{d}>0, \ v^{1-p}\Delta v \in C_0^{2+\alpha}(\overline \om),\ (v^{1-p}\Delta  )^2 v \in C_0^\alpha(\overline \om) \Big\}\nonumber,
\end{align*}
where $0<\al <\min \{1, p-1,\frac{2}{p-1}\},$ and the subscript $0$ in $C_0^{k+\al}$  means that every function belonging to this set vanishes on $\partial\om$.

Step 1. If $u_0\in\mathcal{S}$, then the theorem follows from repeatedly using the short time existence theorem (Theorem 3.2 in \cite{JX19}), \eqref{eq:DKV-a}  and Theorem \ref{thm:uttbound}.

Step 2. Otherwise, one can take a sequence of functions $\{u_0^{(j)}\}$ in $\mathcal{S}$ to approximate $u_0$ in $C_0(\overline\Omega)$. Let $u^{(j)}$ be the weak solution to \eqref{eq:main}-\eqref{eq:main-d} with initial data $u^{(j)}(\cdot, 0)=u_0^{(j)}$. Then it follows from the comparison principle, the H\"older estimates up to $\pa \om $ (see Chen-DiBenedetto \cite{CDi}) and the Ascoli-Arzela theorem, there exists a subsequence of $\{u^{(j)}\}$, which is still denoted as $\{u^{(j)}\}$, such that it converges locally uniformly to $u$ on $\overline\Omega\times(0,T^*)$. Let $\delta\in (0,T^*/4)$. By \eqref{eq:DKV-a}, there exists $C$ independent of $j$ such that 
\begin{equation}\label{eq:appequniformbound}
\frac{d(x)}{C}\le u^{(j)}\le Cd(x)\quad\mbox{in }\Omega\times[\delta, T^*-\delta] \mbox{ for all large }j.
\end{equation}
Meanwhile,  since $u^{(j)}_0\in\mathcal{S}$, then from Step 1, we know $u^{(j)}\in C^{3,2}(\overline\Omega\times[0,T^*-\delta])$ for all large $j$. Then, it follows from \eqref{eq:appequniformbound}  and Theorem \ref{thm:uttbound} that there exists a subsequence of $\{u^{(j)}\}$ converging to $u$ in $C^{3,2}(\overline\Omega\times[2\delta,T^*-\delta])$. In particular, $u(\cdot,2\delta)\in\mathcal{S}$, and thus, the theorem follows from Step 1. 
\end{proof}

The same argument will show that Theorem  \ref{thm:main} also applies to the equation \eqref{eq:FDE-1}.

\smallskip

\noindent T. Jin

\noindent Department of Mathematics, The Hong Kong University of Science and Technology\\
Clear Water Bay, Kowloon, Hong Kong\\
Email: \textsf{tianlingjin@ust.hk}

\medskip

\noindent J. Xiong

\noindent School of Mathematical Sciences, Laboratory of Mathematics and Complex Systems, MOE\\ Beijing Normal University, 
Beijing 100875, China\\
Email: \textsf{jx@bnu.edu.cn}

\end{document}